\documentclass[12pt, a4paper]{amsart}
\usepackage{amsmath}
\usepackage{geometry,amsthm,graphics,tabularx,amssymb,
shapepar}
\usepackage{amscd}
\usepackage[usenames]{color}
\usepackage{tikz-cd}
\usepackage{graphicx}

\usepackage[all]{xy}

\newcommand{\sG}{\mathsf G}

\newcommand{\BC}{{\mathbb {C}}}

\newcommand{\BR}{{\mathbb {R}}}

\newcommand{\Spec}{{\mathrm{Spec}}}

\newcommand{\be}{\begin {equation}}
\newcommand{\ee}{\end {equation}}
\newcommand{\bee}{\begin {equation*}}
\newcommand{\eee}{\end {equation*}}

\theoremstyle{Theorem}

\newtheorem{thm}{Theorem}[section]

\theoremstyle{Theorem}

\theoremstyle{Theorem}
\newtheorem{prp}{Proposition}[section]
\newtheorem{corp}[prp]{Corollary}

\theoremstyle{Plain}
\newtheorem*{remark}{Remark}

\theoremstyle{Definition}
\newtheorem{dfn}{Definition}[section]
\newtheorem{dfnt}[thm]{Definition}

\newtheorem{lemd}[dfn]{Lemma}

\begin{document}

\title[Abelian Nash manifold]{Classification of abelian Nash manifolds}

\author[Y. Bao]{Yixin Bao}

\address{School of Sciences, Harbin Institute of Technology, Shenzhen, 518055, China}
\email{mabaoyixin1984@163.com}

\author [Y. Chen] {Yangyang Chen}

\address{School of Sciences, Harbin Institute of Technology, Shenzhen, 518055, China}
\email{chenyangyang@hit.edu.cn}




\begin{abstract}
By the algebraization of affine Nash groups, a connected affine Nash group is an abelian Nash manifold if and only if its algebraization is a real abelian variety. We first classify real abelian varieties up to isomorphisms. Then with a bit more efforts, we classify abelian Nash manifolds up to Nash equivalences.
\end{abstract}

 \maketitle


\section{Introduction}\label{Introduction}

A Nash group is a group which is simultaneously a Nash manifold such that all group operations are Nash maps. For basic notions concerning  Nash manifolds and Nash maps, we refer the reader to \cite{BCR,Sh87} and \cite[Section 2]{Sun}. A Nash manifold is said to be affine if it is Nash diffeomorphic to a Nash submanifold of some finite dimensional real vector spaces. A Nash group is affine if the underlying Nash manifold is affine. By the work of Hrushovski and Pillay \cite{HP1,HP2}, we know that affine Nash groups are closely related to real algebraic groups. Actually, affine Nash groups are precisely the finite covers of real algebraic groups (see \cite[Theorem 3.3]{FS}). In literature, researchers try to extend the results for real algebraic groups to those for affine Nash groups. For a recent work in this direction, we refer the reader to \cite{Can}.

Recall that a Nash group is said to be almost linear if it admits a Nash representation (finite dimensional real representation such that the action map is a Nash map) with a finite kernel. Structures of almost linear Nash groups were systematically studied by Sun \cite{Sun}. Almost linear Nash groups provide a very convenient setting for the  study of infinite dimensional smooth representations, see \cite{AGKL, AGS, CS, du, SZ}. By \cite[Proposition III.1.7]{Sh87}, we know that every almost linear Nash group is affine.
It is well known that real abelian varieties play an important role in the theory of real algebraic groups. The analog concept of real abelian variety in the setting of affine Nash group is the so called abelian Nash manifold. Recall from \cite[Definition 1.1]{FS} that an affine Nash group is said to be complete if it has no nontrivial connected almost linear Nash subgroup, while an abelian Nash manifold is defined to be a connected complete affine Nash group. Now we state Chevalley's theorem in the setting of affine Nash groups.

\begin{thm}[\cite{FS}, Theorem 1.2]\label{Chevalley}
Let $G$ be a connected affine Nash group. Then there exists a unique connected normal almost linear Nash subgroup $H$ of $G$ such that $G/H$ is an abelian Nash manifold.
\end{thm}

Note that the quotients of affine Nash groups by their normal Nash subgroups are naturally affine Nash groups (see \cite[Proposition 4.1]{FS}). For Chevalley's theorem in the setting of algebraic groups, we refer the reader to the modern approach \cite[Theorem 1.1]{Con}.

The main goal of this article is to give a classification of abelian Nash manifolds. In \cite{MS}, the authors classified all one-dimensional abelian Nash mainfolds. We generalize their results and classify all finite-dimensional abelian Nash manifolds. Our approach to abelian Nash manifolds is based on the algebraization of affine Nash groups, which was introduced in \cite{FS}.


\begin{dfnt}[\cite{FS}, Definition 3.1]\label{algebraization}
Let $G$ be a Nash group. An algebraization of $G$ is an algebraic group $\mathsf G$  over $\BR$, together with a Nash homomorphism $G\rightarrow \mathsf G(\BR)$ which has a finite kernel and whose image is Zariski dense in $\mathsf G$.
\end{dfnt}


Note that for every algebraic group  $\mathsf G$  over $\BR$, $\mathsf G(\BR)$ is  naturally an affine Nash group. The homomorphism $G\rightarrow \mathsf G(\BR)$ of Definition \ref{algebraization} is called the algebraization homomorphism of the algebraization. It turns out that (\cite[Lemma 3.7]{FS}) a connected affine Nash group is an abelian Nash manifold if and only if its algebraization is a real abelian variety. To classify abelian Nash manifolds, we need first classify real abelian varieties. By base change, for any real abelian variety, we get a complex abelian variety, with some additional structure. The classification of complex abelian varieties is well known, see for example \cite[Chapter I, Theorem 2.9]{Mil3}. With a bit more efforts, we classify real abelian varieties and abelian Nash manifolds. Now we state the main results of this paper.

We define the category of real polarizable lattices as follows. In this article, a real lattice is defined to be a pair $(V_{0}, \Lambda)$, where

\begin{itemize}
    \item $V_{0}$ is a real vector space,
    \item $\Lambda$ is a full lattice in $V_{0}\otimes_{\mathbb{R}}\mathbb{C}$,
    \item $\Lambda \subseteq \frac{1}{2}((\Lambda \cap V_{0}) \oplus (\Lambda \cap V_{0}\sqrt{-1}))$.
\end{itemize}
A real polarizable lattice is a real lattice such that there exists a positive definite symmetric bilinear form $S$ on $ V_{0}$ satisfying
\begin{itemize}
    \item $S(y_{1},x_{2})-S(x_{1},y_{2})$ is an integer, where $x_1, x_2, y_1, y_2 \in V_0$ and $x_1+y_1\sqrt{-1}, x_2+y_2\sqrt{-1}\in\Lambda$.
\end{itemize}

A morphism from $(V_{0}, \Lambda)$ to $(V_{0}', \Lambda')$ in the category of real polarizable lattices is defined to be a real linear map from $V_{0}$ to $V_{0}'$ whose complexification maps $\Lambda$ into $\Lambda'$. The first main result of this article gives a classification of real abelian varieties.

\begin{thm}\label{main1}
The isomorphism classes of real abelian varieties are one-to-one corresponding to the isomorphism classes of real polarizable lattices.
\end{thm}

Furthermore, let $(V_{0}, \Lambda)$ and $(V_{0}', \Lambda')$ be two real polarizable lattices. An $\mathbb{R}$-linear isomorphism $\varphi$ from $V_{0}$ to $V_{0}'$ is called an imaginary isogeny if
$\varphi(\Lambda_{+})  =\Lambda_{+}'$ and  $\varphi_{\BC}(\mathbb{Q}\Lambda) = \mathbb{Q}\Lambda'$, where $\Lambda_{+} := \Lambda \cap V_{0}$ and
$\Lambda_{+}' := \Lambda' \cap V_{0}'$. It is obvious that imaginary isogeny is an equivalence relation among real polarizable lattices.

\begin{thm}\label{main2}
 The Nash equivalence classes of abelian Nash manifolds are one-to-one corresponding to the equivalence classes of real polarizable lattices modulo imaginary isogenies.
\end{thm}

This article is organized as follows. In Section \ref{prel}, we recall some basic facts about algebraic groups and algebraizations of affine Nash groups. In next three sections, we introduce the theory of Galois descent and establish connection between real polarized abelian varieties and complex polarized abelian varieties with descent maps. With these preparations, we prove Theorem \ref{main1} in Section \ref{proof}. Using Theorem \ref{main1}, we establish the classification of abelian Nash manifolds in Section \ref{classification}.

{\bf Acknowledgements.} The authors would like to thank Professor Binyong Sun for initiating this work and constant encouragements, and Doctor Yuancao Zhang for his crucial consultancy on knowledge of algebraic geometry. Yixin Bao is supported by the NSFC (Grant No.11801117) and the Natural Science Foundation of Guangdong Province, China (Grant No.2018A030313268).

\section{Preliminaries}\label{prel}

\subsection{Preliminaries on algebraic groups}

For basic notions in algebraic geometry, we refer the reader to \cite{Ha}. In this subsection
we recall some well-known facts concerning algebraic groups. For more details, see \cite[Section 2]{Br}.

Let $k$ be a field of characteristic 0. We consider separated
schemes over $k$ unless otherwise stated; we call them $k$-schemes in this article, or just schemes if no confusion is created. The product of two $k$-schemes is understood to be the fiber product over $\mathrm{Spec}(k)$. Let
$K$ be a field extension of $k$. For each $k$-scheme $X$, we obtain a $K$-scheme $X_K$ by base change.

For each $k$-scheme $S$, the set of morphisms from $S$ to $X$ over $k$ is denoted by $X(S)$, called $S$-points of $X$. When $S$ is affine,
namely, $S = \mathrm{Spec}(R)$ for some $k$-algebra $R$, we also use the notation $X(R)$ for $X(S)$. Particularly, we
have the set $X(k)$ of $k$-rational points. A variety $X$ over $k$ is a geometrically integral $k$-scheme of finite type.

\begin{dfn}\label{gs}
A $k$-group scheme is a group object in the category of $k$-schemes.
\end{dfn}

Let $G$ be a $k$-group scheme. Then it is equipped with two morphisms $m: G\times G\rightarrow G$, $i: G\rightarrow G$ over $k$ and a
$k$-rational point $e\in G(k)$ such that for each $k$-scheme $S$, the set $G(S)$ is a group with the multiplication map $m(S)$, the inverse
map $i(S)$ and the neutral element $e(S)$. Here the maps $m(S)$ and $i(S)$ are canonically induced by that of $m$ and $i$, and the $S$-point $e(S)$ is canonically induced by $e$.
For each $k$-group scheme $G$, we yield a $K$-group scheme $G_K$ by base change.

\begin{dfn}\label{sgs}
Let $G$ be a group scheme. A subgroup scheme of $G$ is a locally closed subscheme $H$ such that $H(S)$ is a subgroup of
$G(S)$ for each scheme $S$. We say that $H$ is normal in $G$, if $H(S)$ is normal in $G(S)$ for each scheme $S$.
\end{dfn}

\begin{dfn}\label{homo}
Let $G, H$ be group schemes. A morphism $f:G\rightarrow H$ is called a homomorphism if
the induced map $f(S): G(S)\rightarrow H(S)$ is a group homomorphism for each scheme $S$.
\end{dfn}

\begin{dfn}\label{ag}
An algebraic group over $k$ is a $k$-group scheme of finite type.
\end{dfn}

This definition is somewhat more general than the classical one, as in \cite[Chpter 1]{Bo}, where an algebraic group over $k$
was defined to be a geometrically reduced $k$-group scheme of finite type. Yet both notions coincide when $k$ is of characteristic 0
(see \cite[Section V.3, Corollary 3.9]{Per}). A group scheme is linear if it is isomorphic to a closed subgroup scheme of $\mathrm{GL}_n$ for some positive integer $n$. 
By an abelian variety, we mean a connected complete algebraic group. For more details about abelian varieties, we refer the reader to \cite{Mu,Mil3}.

\subsection{Preliminaries on algebraizations}

Recall that an algebraization of an affine Nash group $G$ is defined to be an algebraic group $\sG$ over $\BR$ together with a Nash homomorphism
$G\rightarrow\sG(\BR)$, which has a finite kernel and whose image is Zariski dense in $\sG$. In this subsection, we will recall some basic facts about algebraizations that will be used in this paper.

\begin{lemd}[\cite{FS}, Lemma 3.4]\label{unique}
Let $G$ be an affine Nash group. Let $\mathsf G_1, \mathsf G_2$ be
two algebraizations of $G$. Then there exists an algebarization
$\mathsf G_3$ of $G$, and two surjective algebraic homomorphisms
$\mathsf G_3 \rightarrow \mathsf G_1$ and $\mathsf G_3\rightarrow
\mathsf G_2$ with finite kernels such that the the diagram
\[
  \xymatrix{
   G \ar[r] \ar[rd] \ar[d] &\mathsf G_1(\BR) \\
  \mathsf G_2(\BR)  &\mathsf G_3(\BR) \ar[l] \ar[u] .
   }
\]
commutes. Here the three arrows starting from $G$ are the
algebraization homomorphisms.
\end{lemd}

This lemma asserts that the algebraization of an affine Nash group is unique up to coverings.
The following result says that the analogue of a linear real algebraic group in the setting of affine Nash groups is just an almost linear Nash group.

\begin{lemd}[\cite{FS}, Lemma 3.6]\label{alge}
Let $\mathsf G$ be an algebraization of an affine Nash group $G$. Then $G$ is almost linear if and only if $\mathsf G$ is linear as an algebraic group.
\end{lemd}

Recall that an abelian Nah manifold is defined to be a connected complete affine Nash group.

\begin{lemd}[\cite{FS}, Lemma 3.7]\label{alge2}
Let $\mathsf G$ be an algebraization of a connected affine Nash group $G$. Then $G$ is an abelian Nash manifold if and only if $\mathsf G$ is an abelian variety.
\end{lemd}

By this lemma, we know that abelian Nash manifolds are analogues of real abelian varieties.

\section{Galois descent}\label{galois}

Let $\overline{\phantom{A}}$ be the complex conjugation on $\mathbb{C}$. Denote the induced map on $\Spec \mathbb{C}$ also by $\overline{\phantom{A}}$. An $\mathbb{R}$-automorphism $\sigma$ of a complex scheme $X$ is called a descent map if $\sigma^{2}=\mathrm{id}$ and the following diagram
$$\xymatrix@=0.6cm{X \ar[r]^{\sigma} \ar[d] & X \ar[d]\\
   \mathrm{Spec} \mathbb{C}\ar[r]^{\overline{\phantom{A}}} &\mathrm{Spec} \mathbb{C}
     }$$
commutes. In this article, a complex scheme with a fixed descent map is called a descended complex scheme. A morphism between two descended complex schemes is defined to be a morphism over $\mathbb{C}$ which commutes with descent maps.

Let $S$ be a real scheme. Changing the base to $\mathbb{C}$, we obtain a complex scheme $S_{\mathbb{C}}$. Consider the following Cartesian diagram:

$$\xymatrix@=0.6cm{
     S_{\mathbb{C}} \ar[r]^{\sigma} \ar[d] & S_{\mathbb{C}}\ar[d]\\
    \mathrm{Spec} \mathbb{C}\ar[r]^{\overline{\phantom{A}}} & \mathrm{Spec} \mathbb{C}.
}$$
Then we have $\sigma^{2}=\mathrm{id}$ and $S_{\mathbb{C}}$ is a complex scheme with a descent map.

Consider the following two categories:
\begin{itemize}
    \item The category of quasi-projective varieties over $\mathbb{R}$ and their morphisms.
    \item The category of descended quasi-projective varieties over $\mathbb{C}$ and their morphisms.
\end{itemize}
As mentioned above, the base change to $\mathbb{C}$ naturally induces a functor from the first category to the second category. Moreover, we have the following proposition.

\begin{prp}\label{des1}
The base change naturally induces an equivalence between the category of real quasi-projective varieties and the category of descended complex quasi-projective varieties.
\end{prp}

\begin{proof}
It follows directly from 
\cite[Lemma 2.12]{Ja} or \cite[Corollary 16.25]{Mil1}.
\end{proof}

\section{Descent from projective complex manifolds}\label{projective}

Let $V$ be a smooth projective variety over $\BC$ with a descent map $\sigma$. Let $V^{\mathrm{an}}$ be its $\mathbb{C}$-points with the usual topology. Then $V^{\mathrm{an}}$ has a canonical complex manifold structure which was called the analytification of $V$. The descent map $\sigma$ induces an anti-holomorphic autohomeomorphism $\sigma^{\mathrm{an}}$ of $V^{\mathrm{an}}$. An anti-holomorphic autohomeomorphism on a complex manifold is called a descent map if its square is the identity map, while a complex manifold with a fixed descent map is called a descended complex manifold. Morphisms between two descended complex manifolds are those holomorphic homomorphisms which commute with the corresponding descent maps. 

It follows from GAGA \cite{Se} that the analytification induces an equivalence from the category of smooth complex projective varieties to the category of complex projective manifolds. Furthermore, we have the following proposition.

\begin{prp}\label{des2}
The analytification induces an equivalence between the category of descended smooth complex projective varieties and the category of descended complex projective manifolds.
\end{prp}

\begin{proof}
We only need to prove a descent map on the analytification comes from an underlying descent map, and a morphism between descended complex projective manifolds comes from a morphism between underlying descended complex projective varieties.

For a smooth complex projective variety $V$, let $\overline{V}$ denote the pullback of $V$ along the complex conjugation. Then we have the following diagram
$$\xymatrix@=0.6cm{
  \overline{V} \ar[r]^{\overline{\phantom{A}}} \ar[d] & V\ar[d]\\
   \mathrm{Spec} \mathbb{C}\ar[r]^{\overline{\phantom{A}}} &\mathrm{Spec} \mathbb{C}.   }$$
Passing to the complex manifolds, we have
$$\xymatrix{ \overline{V}^{\mathrm{an}} \ar[r]^{{\overline{\phantom{A}}}^{\mathrm{an}}} & V^{\mathrm{an}}.
    }$$
The map ${\overline{\phantom{A}}}^{\mathrm{an}}$ is an anti-holomorphic homeomorphism.

Let $\theta$ be a descent map on $V^{\mathrm{an}}$. Composing it with ${\overline{\phantom{A}}}^{\mathrm{an}}$, we get an holomorphic homeomorphism between ${\overline{V}}^{\mathrm{an}}$ and $V^{\mathrm{an}}$. By GAGA \cite{Se}, it is induced by an isomorphism $\varphi$ of $\mathbb{C}$-varieties, i.e., $\varphi^{\mathrm{an}}= \theta\circ{\overline{\phantom{A}}}^{\mathrm{an}}$. Composing it with $\overline{\phantom{A}}^{-1}$, we get the following diagram
$$\xymatrix@=0.6cm{ V \ar[r]^{\sigma} \ar[d] & V\ar[d]\\
   \mathrm{Spec} \mathbb{C}\ar[r]^{\overline{\phantom{A}}} &\mathrm{Spec} \mathbb{C},
     }$$
where $\sigma=\varphi\circ\overline{\phantom{A}}^{-1}$. Then we have $\theta=\varphi^{\mathrm{an}}\circ(\overline{\phantom{A}}^{-1})^{\mathrm{an}}=\sigma^{\mathrm{an}}$.
By GAGA \cite{Se}, $\sigma^{2}=\mathrm{id}$. Thus $\sigma$ is a descent map on $V$ and $\theta$ is the analytification of $\sigma$. Similarly, we may prove that a morphism between two descended complex projective manifolds comes from a morphism between the underlying descended complex projective varieties.
\end{proof}

\section{Dual Abelian Varieties and Polarizations}\label{Dual And Polarization}

Let $A$ be an abelian variety over a field $k$. Its dual abelian variety $\widehat{A}$ over $k$ is defined to be the identity connected component of its Picard group. A homomorphism $\phi$ from $A$ to $\widehat{A}$ is called a polarization on $A$ if $\phi$ is self-dual and the pullback of the Poincare bundle on $A\times \widehat{A}$ through $\Gamma(\phi)$ is an ample line bundle on $A$, where $\Gamma(\phi)$ is the graph homomorphism of $\phi$.

Now we consider a polarization on a real abelian variety $A$. By Proposition \ref{des1}, we obtain a polarization on the complex abelian variety $A_{\mathbb{C}}$ which commutes with the corresponding descent maps. By Proposition \ref{des2}, we obtain a polarization on the complex manifold $A_{\mathbb{C}}^{\mathrm{an}}$. By Lefschetz Theorem \cite[Chapter 1, Section 3]{Mu}, $A_{\mathbb{C}}^{\mathrm{an}}$ is isomorphic to $V/\Lambda$, where $V$ is a finite-dimensional complex vector space and $\Lambda$ is a full lattice in $V$. On the other hand, the analytification of the dual abelian variety $\widehat{A}_{\mathbb{C}}$ is canonically isomorphic to $\widehat{V}/\widehat{\Lambda}$, where $\widehat{V}$ is the space of anti-linear functions on $V$ and $\widehat{\Lambda}$ is the lattice consisting of those functions whose imaginary parts take integer values on $\Lambda$. The descent map $\sigma^{\mathrm{an}}$ on $A_{\mathbb{C}}^{\mathrm{an}}$ induced by $A$ is an anti-holomorphic autohomeomorphism which commutes with the group multiplication and preserves the identity. Thus $\sigma^{\mathrm{an}}$ is an anti-linear involution on $V$ and $\Lambda$ is stable under $\sigma^{\mathrm{an}}$. Each polarization on $A_{\mathbb{C}}^{\mathrm{an}}$ is of the form $\phi_{H}$, where $H$ is a positive definite Hermitian form on $V$ such that $E=\mathrm{im}H$ takes integer values on $\Lambda$ and the homomorphism $\phi_{H}$ maps $v\in V$ to the anti-linear function $H(v,)$ on $V$. As mentioned above, a polarization on $A_{\mathbb{C}}^{\mathrm{an}}$ induced from a polarization on $A$ commutes with the corresponding descent maps. To understand these polarizations, we need to describe the descent map $\widehat{\sigma}^{\mathrm{an}}$ on $\widehat{A}_{\mathbb{C}}^{\mathrm{an}}$ induced from $\widehat{A}$ explicitly.

As mentioned in Section \ref{projective}, the descent map $\sigma$ on $A_{\mathbb{C}}$ induces a map $\widehat{\sigma}$ on the Picard group $\mathrm{Pic}A_{\mathbb{C}}$, which maps a line bundle $L$ on $A_{\mathbb{C}}$ to its pullback $\overline{L}$ as defined by the following Cartesian diagram
\begin{equation*}
  \xymatrix@=0.6cm{
    \overline{L} \ar[d] \ar[r] & L \ar[d] \\
    A_{\mathbb{C}} \ar[d] \ar[r]^{\sigma} & A_{\mathbb{C}} \ar[d] \\
    \mathrm{Spec}\mathbb{C} \ar[r]^{\overline{\phantom{A}}} & \mathrm{Spec}\mathbb{C}.   }
\end{equation*}
Passing to analytification, we obtain the following diagram
\begin{equation*}
  \xymatrix@=0.6cm{
    \overline{L}^{\mathrm{an}} \ar[d] \ar[r] & L^{\mathrm{an}} \ar[d] \\
    A_{\mathbb{C}}^{\mathrm{an}} \ar[r]^{\sigma^{\mathrm{an}}} & A_{\mathbb{C}}^{\mathrm{an}},   }
\end{equation*}
where the vertical maps are holomorphic and the horizontal maps anti-holomorphic homeomorphisms. Furthermore, the horizontal map between complex bundles is anti-linear. Hence we obtain a map $\widehat{\sigma}^{\mathrm{an}}$, which is the analytification of $\widehat{\sigma}$ and maps the holomorphic bundle $L^{\mathrm{an}}$ to $\overline{L}^{\mathrm{an}}$.

By Appell-Humbert Theorem \cite[Chapter 1, Section 2]{Mu}, each line bundle in $\widehat{A}_{\mathbb{C}}^{\mathrm{an}}$ is of the form $L(\alpha)$, where $\alpha$ is an anti-linear function on $V$ and $L(\alpha)$ is the quotient of $V\times\mathbb{C}$ by the action of $\Lambda$ given by
\begin{equation*}
  \lambda\cdot(v,z)=(v+\lambda,e^{\pi(\alpha(\lambda)-\overline{\alpha(\lambda)})}z), \quad \forall \lambda\in\Lambda.
\end{equation*}
By the universal property of the descent map $\widehat{\sigma}^{\mathrm{an}}$, it is routine to check that $\widehat{\sigma}^{\mathrm{an}}$ maps $L(\alpha)$ to $L(\widehat{\sigma}(\alpha))$, where $\widehat{\sigma}(\alpha)$ is also an anti-linear function on $V$ such that $\widehat{\sigma}(\alpha)(v)=\overline{\alpha(\sigma^{\mathrm{an}} v)}$ for each $v\in V$. It is easy to check that $\phi_{H}\circ\sigma^{\mathrm{an}}=\widehat{\sigma}^{\mathrm{an}}\circ\phi_{H}$ if and only if $H(v_{1},v_{2})=\overline{H(\sigma^{\mathrm{an}}(v_{1}),\sigma^{\mathrm{an}}(v_{2}))}$ for $v_{1},v_{2}\in V$.

\section{Proof of Theorem \ref{main1}}\label{proof}


We consider the following category, which we call the category of descended complex polarizable lattices. The objects in this category are the following triples $(V, \Lambda, \theta)$ where
\begin{itemize}
    \item $V$ is a complex vector space,
    \item $\theta$ is an anti-linear involution of $V$,
    \item $\Lambda$ is an $\theta$-stable full lattice in $V$,
\end{itemize}
such that there exists a positive definite Hermitian form $H$ on $V$ satisfying
\begin{itemize}
    \item $E = \mathrm{im} H$ takes integer values on $\Lambda$,
    \item $\overline{H(\theta(v_{1}),\theta(v_{2}))}=H(v_{1},v_{2})$.
\end{itemize}
A morphism from $(V, \Lambda, \theta)$ to $(V', \Lambda', \theta')$ is a complex linear map $f: V\rightarrow V'$ satisfying
\begin{equation*}
  f\circ\theta=\theta'\circ f, \quad f(\Lambda)\subset\Lambda'.
\end{equation*}

\begin{remark}
In fact, the second condition on the positive definite Hermitian form is not necessary for the definition of descended complex polarizable lattices. Given a form $H$ on $V$ satisfying the first condition, we construct a form $\mathbf{H}$ such that
\begin{equation*}
  \mathbf{H}(v_{1},v_{2})=H(v_{1},v_{2})+\overline{H(\theta(v_{1}),\theta(v_{2}))}, \quad \forall v_{1},v_{2}\in V.
\end{equation*}
Then $\mathbf{H}$ naturally satisfies the second condition. We state the second condition to emphasize that the Hermitian form corresponds to a polarization induced from the underlying real abelian variety.
\end{remark}

\begin{prp}\label{equivalence1}
The composition of the analytification and the base change induces an equivalence between the category of real abelian varieties and the category of descended complex polarizable lattices.
\end{prp}
\begin{proof}
Given a real abelian variety $A$ and a polarization on $A$. By the results in Section \ref{Dual And Polarization}, we know that the complex manifold $A_{\mathbb{C}}^{\mathrm{an}}$ is isomorphic to $V/\Lambda$ for some complex vector space $V$ and a full lattice $\Lambda$ in $V$, the descent map $\sigma^{\mathrm{an}}$ is an anti-linear Lie group homomorphism on $A_{\mathbb{C}}^{\mathrm{an}}$ and the polarization on $A$ corresponds to a positive definite Hermitian form $H$ on $V$ satisfying the conditions on the category of descended complex polarizable lattices. Furthermore, an algebraic group homomorphism between two real abelian varieties induce a Lie group homomorphism commuting with the corresponding involutions.

Conversely, for any object $(V, \Lambda,\theta)$ in the category of descended complex polarizable lattices, we know that $V/\Lambda$ is isomorphic to the analytification of a complex abelian variety by Lefschetz Theorem. By Propositions \ref{des1} and \ref{des2}, we know this complex abelian variety is the complexification $A_{\mathbb{C}}$ of a real projective variety $A$. Furthermore, since $\theta$ is an anti-linear map on $V$, it is natural to check that $\theta$ commutes with the group operations on $V/\Lambda$. By Propositions \ref{des1} and \ref{des2}, it is routine to check that the group structure on $V/\Lambda$ is induced from a group structure on $A$. So $A$ is a real abelian variety. Similarly, a morphism in the category of descended complex polarizable lattices is induced from an algebraic group homomorphism between two real abelian varieties.
\end{proof}

Recall the definition of the category of real polarizable lattices in Section \ref{Introduction}. We have the following proposition.
\begin{prp}\label{equivalence2}
The category of descended complex polarizable lattices is equivalent to the category of real polarizable lattices.
\end{prp}
\begin{proof}
Let $(V, \Lambda, \theta)$ be an object in the category of descended complex polarizable lattices with $\dim_{\mathbb{C}} V = g$. Let $V_{+}$ (resp. $\Lambda_{+}$) be the $\theta$-fixed part of $V$ (resp. $\Lambda$) and $V_{-}$ (resp. $\Lambda_{-}$) the set of elements of $V$ (resp. $\Lambda$) on which $\theta$ acts by $-1$. Since $\theta^{2}=\mathrm{id}$ , we have
$$\mathbb{Q}\Lambda=\mathbb{Q}\Lambda_{+} \oplus  \mathbb{Q}\Lambda_{-}.$$
So we have
$$\mathrm{rank}_{\mathbb{Z}}\Lambda_{+}+\mathrm{rank}_{\mathbb{Z}}\Lambda_{-} = 2g.$$
Furthermore , we have
$$\Lambda_{+} \oplus \Lambda_{-} \subseteq \Lambda \subseteq \frac{1}{2}(\Lambda_{+} \oplus \Lambda_{-}).$$

For convenience, we also denote $V_{+}$ by $V_{0}$. Multiplication by $\sqrt{-1}$ gives a bijective map from $V_{+}$ to $V_{-}$. So $V_{0}$ is a real vector space of dimension $g$ and $V_{-}=\sqrt{-1}V_{0}$.

Obviously
$$\mathbb{R}\Lambda_{+} \subseteq V_{+}, \quad \mathbb{R}\Lambda_{-} \subseteq V_{-},$$
thus
$$\mathrm{rank}_{\mathbb{Z}}\Lambda_{+} \leq g, \mathrm{rank}_{\mathbb{Z}}\Lambda_{-} \leq g,$$
which forces that
$$\mathrm{rank}_{\mathbb{Z}}\Lambda_{+} = \mathrm{rank}_{\mathbb{Z}}\Lambda_{-} = g.$$
Hence
$$V_{0} =\mathbb{R} \Lambda_{+},\quad  V_{-} = \mathbb{R}\Lambda_{-}.$$

Let $H$ be a positive definite Hermitian form satisfying conditions on the category of descended complex polarizable lattices. We construct a bilinear form $S$ on $V_{0}$ such that
$$S(x ,y ) = E (x\sqrt{-1}, y), \quad \forall x,y\in V_{0}.$$
Here $E$ is the imaginary part of $H$. Then $S$ is positive definite and symmetric. It is routine to check that $S$ satisfies those conditions on the category of real polarizable lattices. We have a natural functor which maps the object $(V, \Lambda, \theta)$ to the object $(V_{0}, \Lambda)$. It is routine to check that this functor is an equivalence of categories.
\end{proof}

Combining Propositions \ref{equivalence1} and \ref{equivalence2}, we finish the proof of Theorem \ref{main1}.


For a real abelian variety $A$, let $(V, \Lambda, \theta)$ be the descended complex polarizable lattice associated with $A$. Then the manifold $A(\mathbb{R})$ consists of the $\theta$-fixed points of $V/\Lambda$.

\begin{prp}\label{con}
    The group of connected components of $A(\mathbb{R})$ is canonically isomorphic to $\frac{1}{2}(\Lambda_{+}\oplus \Lambda_{-})/(\Lambda + \frac{1}{2}\Lambda_{+})$ and the identity connected component of $A(\mathbb{R})$ is canonically isomorphic to $V_{0}/\Lambda_{+}$.
\end{prp}

\begin{proof}
Without creating confusion, let $\theta$ denote the anti-linear involution on $V$ and the 2-cyclic group generated by the involution. From the exact sequence
\begin{equation*}
  \xymatrix{
    0 \ar[r] & \Lambda \ar[r] & V \ar[r] & V/\Lambda \ar[r] & 0,   }
\end{equation*}
we obtain the following exact sequence
\begin{equation*}
  \xymatrix{
     0 \ar[r] & H^{0}(\theta,\Lambda) \ar[r] & H^{0}(\theta,V) \ar[r] & H^{0}(\theta,V/\Lambda) \ar[r] & H^{1}(\theta,\Lambda) \ar[r] & H^{1}(\theta,V),   }
\end{equation*}
i.e.,
\begin{equation*}
  \xymatrix{
     0 \ar[r] & \Lambda_{+} \ar[r] & V_{0} \ar[r] & A(\mathbb{R}) \ar[r] & H^{1}(\theta,\Lambda) \ar[r] & H^{1}(\theta,V).   }
\end{equation*}
Since $V$ is a $\mathbb{Q}$-linear space, $H^{1}(\theta,V)=0$. Hence we have the exact sequence
\begin{equation*}
  \xymatrix{
    0 \ar[r] &  V_{0}/\Lambda_{+} \ar[r] & A(\mathbb{R}) \ar[r] & H^{1}(\theta,\Lambda) \ar[r] & 0.   }
\end{equation*}
Since $\theta$ is 2-cyclic, we know that $H^{1}(\theta,\Lambda)$ is isomorphic to $\Lambda_{-}/(\theta-1)\Lambda$. We consider the map from $\frac{1}{2}(\Lambda_{+}\oplus \Lambda_{-})$ to $\Lambda_{-}/(\theta-1)\Lambda$, which maps $\frac{1}{2}(x,y)$ to $y$ for $x\in\Lambda_{+}$, $y\in\Lambda_{-}$. It is routine to check this map induces an isomorphism from $\frac{1}{2}(\Lambda_{+}\oplus \Lambda_{-})/(\Lambda + \frac{1}{2}\Lambda_{+})$ to $\Lambda_{-}/(\theta-1)\Lambda$. Then we finish the proof.
\end{proof}

\begin{corp}
    The group of connected components of $A(\mathbb{R})$ is a $2$-torsion group of order $\leq 2^{g}$, where $g = \dim_{\BR}V_{0}$.
\end{corp}
\begin{proof}
It follows directly from Proposition \ref{con}.
\end{proof}

\begin{remark}
For one-dimensional real abelian varieties, the diamond-shaped lattices mentioned in \cite{MS} correspond to those connected $A(\mathbb{R})$, while the rectangular lattices in \cite{MS} correspond to those $A(\mathbb{R})$ with 2 connected components.
\end{remark}

\section{Application to classification of abelian Nash manifolds}\label{classification}

For two real abelian varieties $A$ and $A'$, let $(V_{0}, \Lambda)$ and $(V_{0}', \Lambda')$ be the real polarizable lattices associated with them respectively. An $\mathbb{R}$-linear isomorphism $\varphi$ from $V_{0}$ to $V_{0}'$ is called an imaginary isogeny between $A$ and $A'$ if $\varphi(\Lambda_{+})  =\Lambda_{+}'$ and  $\varphi_{\mathbb{C}}(\mathbb{Q}\Lambda) = \mathbb{Q}\Lambda'$. By Lemma \ref{alge2}, the identity connected components of $A(\mathbb{R})$ and $A'(\mathbb{R})$ are naturally abelian Nash manifolds. For convenience, we also denote these two Nash groups by $T$ and $T'$.


\begin{prp}\label{1}
    If there is an imaginary isogeny between $A$ and $A'$, then $T$ and $T'$ are Nash equivalent.
\end{prp}
\begin{proof}
Let $(V_{0}, \Lambda)$ and $(V_{0}', \Lambda')$ be the real polarizable lattices associated with $A$ and $A'$ respectively. It is easy to check that $(V_{0}, \Lambda_{+}\oplus \Lambda_{-})$ is also a real polarizable lattices and the real points set in the associated real abelian variety is a finite cover of $T$. Hence its identity component has the same Nash structure with $T$. So we may assume $\Lambda = \Lambda_{+}\oplus \Lambda_{-}$. We identify $V_{0}$ with $V_{0}'$ using the imaginary isogeny. Then we may assume $\Lambda' = \Lambda_{+}\oplus \Lambda_{-}'$. Let $\Lambda_{1} := \Lambda' \cap \Lambda = \Lambda_{+}\oplus (\Lambda_{-} \cap \Lambda_{-}')$. We know that $(V_0, \Lambda_{1})$ is also a polarizable lattice. By Proposition \ref{con}, both $T$ and $T'$ are Nash equivalent to the identity component of the real points set in the real abelian variety associated with $(V_0, \Lambda_{1})$. So $T$ and $T'$ are Nash equivalent.
\end{proof}

\begin{prp}\label{2}
    A Nash equivalence from $T$ to $T'$ is given by an imaginary isogeny.
\end{prp}
\begin{proof}
We identify the real Lie algebras of $T$ and $T'$, i.e. $V_{0}$ and $V'_{0}$. So $V_{0}$ is the universal cover of $T$ and $T'$. By Lemma 1 of \cite{MS}, we endow $V_{0}$ (also $V'_{0}$) a Nash group structure. We have the following diagram
\begin{equation*}
  \xymatrix{
    V_{0} \ar[d]_{\pi} \ar[r] & V=V_{0}\otimes_{\mathbb{R}}\mathbb{C} \ar[d] \\
    T \ar[r] & A_{\mathbb{C}}^{\mathrm{an}},   }
\end{equation*}
where the vertical maps are covering maps and the horizontal maps are natural embeddings. Since the right vertical map is holomorphic, then a meromorphic function on $A_{\mathbb{C}}^{\mathrm{an}}$ is pulled back to a meromorphic function on $V$ which is invariant under $\Lambda$. Hence the meromorphic function field on $A_{\mathbb{C}}^{\mathrm{an}}$ is a subfield of the meromorphic function field on $V$, whose elements are invariant under $\Lambda$. By GAGA \cite{Se}, the meromorphic function field on $A_{\mathbb{C}}^{\mathrm{an}}$ is the rational function field $K(A_{\mathbb{C}})=K(A)\otimes_{\mathbb{R}}\mathbb{C}$. Similarly, we identify the rational function field $K(A'_{\mathbb{C}})$ with a subfield of the meromorphic function field on $V$, whose elements are invariant under $\Lambda'$.

Choose an affine open set $U$ (resp. $U'$) of $A$ (resp. $A'$) such that $U(\mathbb{R})$ (resp. $U'(\mathbb{R})$) contains $0$. Let $R$ and $R'$ be the coordinate ring of $U$ and $U'$, respectively. Let $\pi^{-1}(U(\mathbb{R}))$ (resp. $\pi'^{-1}(U'(\mathbb{R}))$) be the inverse image of $U(\mathbb{R})$ (resp. $U'(\mathbb{R})$) in $V_{0}$. We may shrink the open set $\pi^{-1}(U(\mathbb{R}))\cap\pi'^{-1}(U'(\mathbb{R}))$ to a small open subset $U''$ such that there exists a family of functions in $R$ constituting a set of Nash coordinate functions on $U''$. Then any function in $R'$ satisfies algebraic relations with coefficients in $R$ on $U''$. Since functions in $R$ and $R'$ could be viewed as analytic functions in a open subset of $V$ containing $U''$, the relations also hold in this open subset of $V$. Since functions in $R$ and $R'$ can be extended to meromoprhic functions on $V$, the relations hold on the whole space $V$ (see \cite[Section 4.5, Proposition 8]{Na}). We know that $K(A_{\mathbb{C}})$ (resp. $K(A'_{\mathbb{C}})$) is the quotient field of $R$ (resp. $R'$). Then the composition field $K(A'_{\mathbb{C}})K(A_{\mathbb{C}})$ is algebraic over $K(A_{\mathbb{C}})$. Since $K(A'_{\mathbb{C}})$ is finitely generated, the extension is finite.

Since $\Lambda$ and $\Lambda'$ act commutatively on the meromorphic functions on $V$, $K(A'_{\mathbb{C}})$ is invariant under the action of $\Lambda$. Thus we obtain a natural homomorphism from $\Lambda$ to the Galois group $\mathrm{Gal}(K(A'_{\mathbb{C}})K(A_{\mathbb{C}})/K(A_{\mathbb{C}}))$. Since $K(A'_{\mathbb{C}})K(A_{\mathbb{C}})$ is finite over $K(A_{\mathbb{C}})$, the kernel $\Lambda''$ of this homomorphism is of finite index in $\Lambda$.

Now we prove $\Lambda''\subseteq \Lambda'$. If there exists $v$ such that $v\in\Lambda''$ and $v\notin\Lambda'$, then the image of $v$ in $A'_{\mathbb{C}}$ is not $0$. Since $A'_{\mathbb{C}}$ is projective, there exists a function $f\in K(A'_{\mathbb{C}})$ such that $f(v)\neq f(0)$. It contradicts with the definition of $\Lambda''$. Then $\Lambda''\subseteq \Lambda'$. Comparing the ranks, we have
\begin{equation*}
  \mathbb{Q}\Lambda=\mathbb{Q}\Lambda''=\mathbb{Q}\Lambda'.
\end{equation*}
So the Nash equivalence is given by an imaginary isogeny.

\end{proof}

\begin{prp}\label{3}
    Any abelian Nash manifold is Nash equivalent to some $T$.
\end{prp}
\begin{proof}
By Lemma \ref{alge2}, a quotient of an abelian Nash manifold $M$ by a finite subgroup is Nash equivalent to some $T$. By Proposition \ref{con}, we may assume $T=V_{0}/\Lambda_{+}$. Then consider the pullback Nash structure on $V_{0}$. Then $V_{0}$ is the universal cover of $M$. By \cite[Proposition 2.15]{Sun}, we have a Nash map from $V_{0}$ to $M$. Denote the kernel by $\Lambda_{1}$. Consider the real abelian variety $A_{1}$ associated with the real polarizable lattice $(V_{0}, \Lambda_{-}\oplus \Lambda_{1})$. Let $T_{1}$ be the the identity connected component of the set of real points in $A_{1}$. Then the maps from $V_{0}$ to $M$ and $T$ induce maps from $T_{1}$ to $M$ and $T$. The first is a homeomorphism and the second is a local Nash homeomorphism. By \cite[Proposition 2.15]{Sun}, $M$ and $T_{1}$ are Nash equivalent.
\end{proof}

Combining Propositions \ref{1}, \ref{2} and \ref{3}, we finish the proof of Theorem \ref{main2}.

\begin{remark}
For the one-dimensional case, it is easy to check that a descended complex polarizable lattice $(V,\Lambda,\theta)$ is isomorphic to a rectangular lattice $(\mathbb{C},\Lambda_{0}(\alpha),\overline{\phantom{A}})$ or a diamond lattice $(\mathbb{C},\Lambda_{1}(\alpha),\overline{\phantom{A}})$, where $\Lambda_{0}(\alpha)=\mathbb{Z}+\sqrt{-1}\alpha\mathbb{Z}$ and $\Lambda_{1}(\alpha)=\mathbb{Z}+(\frac{1}{2}+\sqrt{-1}\alpha)\mathbb{Z}$. It follows easily that the lattice $(\mathbb{C},\Lambda_{0}(\alpha),\overline{\phantom{A}})$ is imaginary isogenous to the lattice $(\mathbb{C},\Lambda_{1}(\alpha),\overline{\phantom{A}})$. Moreover, the lattice $(\mathbb{C},\Lambda_{0}(\alpha),\overline{\phantom{A}})$ is imaginary isogenous to the lattice $(\mathbb{C},\Lambda_{0}(\beta),\overline{\phantom{A}})$ if and only if $\alpha/\beta$ is a rational number. Thus we give a classification of one-dimensional abelian Nash manifolds, which corresponds to the sixth case of the main theorem in Section 4 of \cite{MS}.
\end{remark}

\end{document}